\documentclass[oneside,reqno]{amsart} 

\usepackage{amsmath,amssymb,amscd,amsthm}
\usepackage{comment}
\usepackage{enumerate}
\usepackage[all]{xy}
\usepackage{enumitem}

\newtheorem{lemma}{Lemma}

\newtheorem{prop}[lemma]{Proposition}
\newtheorem{cor}[lemma]{Corollary}
\newtheorem{con}[lemma]{Conjecture}
\newtheorem{defi}[lemma]{Definition}
\newtheorem{thm}[lemma]{Theorem}
\newtheorem{ex}[lemma]{Example}
\newtheorem{rmk}[lemma]{Remark}

\newcommand\A{{\mathbb A}}
\newcommand\DD{{{\mathcal V}_X}}
\newcommand{\bb}[1]{\mathbb{#1}}
\newcommand{\fr}[1]{\mathfrak{#1}}

\newcommand{\ra}{\rightarrow}
\newcommand{\D}{\mathcal{V}}
\newcommand{\ol}[1]{\overline{#1}}
\newcommand{\Der}{\mathrm{Der}}
\newcommand{\Ind}{\mathrm{Ind}}
\newcommand{\Res}{\mathrm{Res}}
\newcommand{\Hom}{\mathrm{Hom}}
\newcommand{\Sp}{\mathbb{S}}
\newcommand{\del}{\partial}
\newcommand{\U}{\mathcal{U}}

\DeclareMathOperator{\rank}{rank}

\begin{document}

\title{Representations of Lie algebras of vector fields on affine varieties}
\author{Yuly Billig}
\address{School of Mathematics and Statistics, Carleton University, Ottawa, Canada}
\email{billig@math.carleton.ca}
\author{Vyacheslav Futorny}
\address{ Instituto de Matem\'atica e Estat\'\i stica,
Universidade de S\~ao Paulo,  S\~ao Paulo,
 Brasil}
 \email{futorny@ime.usp.br}
\author{Jonathan Nilsson}
\address{Department of Mathematics, \"{O}rebro University, \"{O}rebro, Sweden}
\email{jonathan.nilsson@oru.se}
\let\thefootnote\relax\footnotetext{{\it 2010 Mathematics Subject Classification.}
Primary 17B20, 17B66; Secondary 13N15.}
\date{}

\begin{abstract} For an irreducible affine variety $X$ over an algebraically closed field of characteristic zero we define two new classes of modules over the Lie algebra of vector fields on $X$ - gauge modules and Rudakov modules, which admit a compatible action of the algebra of functions. Gauge modules are generalizations of  modules of tensor densities whose construction was inspired by non-abelian gauge theory. Rudakov modules are generalizations of a family of induced modules over the Lie algebra of derivations of a polynomial ring studied by Rudakov \cite{Ru}. 
We prove general simplicity theorems for these two types of modules and establish  a pairing between them. 
\end{abstract}

\maketitle

\section*{Introduction}

Classification of complex simple finite-dimensional Lie algebras by Killing (1889), \cite{K} and Cartan (1894) \cite{C1} shaped the development of Lie theory in the first half of the 20th century.  
Since Sophus Lie, the Lie groups and corresponding Lie algebras (as infinitesimal transformations) were related to the symmetries of geometric structures which need not be finite-dimensional.  
Later infinite-dimensional Lie groups and algebras were connected with the symmetries of systems which have an infinite
number of independent degrees of freedom, for example in Conformal
Field Theory, e.g. \cite{BPZ}, \cite{TUY}.  The discovery of first four classes of simple infinite-dimensional Lie algebras goes back to Sophus Lie who introduced certain 
pseudogroups of transformations in small dimensions. 
This work was completed by Cartan who showed that corresponding simple Lie algebras are of type  
$W_n$, $S_n$, $H_n$ and $K_n$ \cite{C2}. These four classes of Cartan type algebras were the first examples of simple infinite-dimensional Lie algebras. The general theory of simple infinite-dimensional Lie algebras at large is still undeveloped, in particular their representation theory.  

The first Witt algebra ${\mathcal W}_1$ is 
the Lie algebra of polynomial vector fields on a circle whose
 universal central extension is the famous Virasoro algebra which plays a crucial role in quantum field theory.   
Mathieu \cite{Ma} classified irreducible modules with finite-dimensional weight spaces for the first Witt algebra ${\mathcal W}_1$.
Higher rank Witt algebras ${\mathcal W}_n$ are 
%
simple Lie algebras of polynomial vector fields on $n$-dimensional torus.   
 Significant efforts were required to generalize the results of Mathieu for an arbitrary $n$: \cite{She}, \cite{La}, \cite{E1}, \cite{E2}, \cite{Bi}, \cite{MZ}, \cite{BF1}, \cite{BF2}, 
resulting  in the classification of simple  weight modules  with finite-dimensional
weight spaces in \cite{BF2}.  We note that 
 understanding of the representations of ${\mathcal W}_n$ is also important for the representation theory of \emph{toroidal} Lie algebras \cite{Bil}.

  The first Witt algebra can be realized as the algebra of  meromorphic
vector fields on the Riemann sphere $\mathbb P^1(\mathbb C)$ which are holomorphic outside of $0$ 
and $\infty$.  This realization has a natural generalization to the case of general Riemann
surfaces.  With a compact Riemann
surface $\Sigma$  and  a finite
subset $S$ of $\Sigma$  one can associate the  Krichever-Novikov type vector field Lie algebra   \cite{KN}  of those meromorphic vector fields on $\Sigma$ which are holomorphic
outside of $S$.   Genus zero case and $S=\{0, \infty\}$ corresponds to the first Witt algebra. 
 For the theory of  Krichever-Novikov Lie algebras  we refer to 
  \cite{Sc}.

For an arbitrary irreducible affine variety $X \subset \A_{\bf k}^n$ over an algebraically closed field $\bf k$ of characteristic $0$ the  Lie algebra $\DD$  of polynomial vector fields 
was studied in  \cite{J1}, \cite{J2}, \cite{Si} (see also \cite{BF3}).  It 
can be identified with the derivation algebra of the coordinate ring of $X$. This algebra
  is not simple in general. In fact, $\DD$  is simple if and only if $X$ is a smooth variety. 
In this paper we begin a systematic study of representations of the Lie algebras $\DD$
for  arbitrary smooth affine varieties $X$ aiming to generalize successful representation theory of Witt algebras.    Note that the general case is significantly more complicated since
 the standard tools of Lie theory, like Cartan subalgebras, root decompositions etc. do not apply in this case. Even in the case when $X$ is an affine elliptic curve, the Lie algebra
$\DD$  does not contain non-zero semisimple or nilpotent elements \cite{BF3}. 
 As a result it was not even clear how to approach the problem of classifying or at least constructing simple modules over $\DD$ for a general smooth variety $X$, since techniques of the classical representation theory of simple finite-dimensional Lie algebras can not be used. Representations of $\DD$ when  $X$ is
an affine space $\A_{\bf k}^n$ were studied by Rudakov \cite{Ru} but the classification problem of simple modules is still open (cf. \cite{LLZ}, \cite{PS}, \cite{CG}). 
  The case of a sphere  $X=\Sp^2$ was treated in \cite{BN}  (see also \cite{BF3}).  A new class of 
modules (\emph{tensor modules}) was constructed, these are modules of tensor fields on a sphere.
  The main idea of \cite{BN} which goes back to \cite{E2} suggests that as a first step one needs to study the category of representations for the Lie algebra $\DD$ which  admit a compatible action of the algebra $A_X$ of polynomial functions on the variety $X$. 
  This idea was successfully implemented in \cite{BF2} in the case of $n$-dimensional torus and led to the classification of simple weight modules with finite-dimensional weight spaces. 
  
  Developing  ideas of \cite{BN} we introduce a category $A\mathcal{V}$-Mod of $\DD$-modules with a compatible action of  $A_X$ for an arbitrary irreducible affine variety $X$. The main goal of the paper is the construction of two families of simple objects in $A\mathcal{V}$-Mod: \emph{gauge modules} and \emph{Rudakov modules}.  

   Rudakov modules $R_p (U)$ are generalizations of certain induced modules over the Lie algebra of derivations of a polynomial ring studied by Rudakov \cite{Ru}. 
These modules are associated with a point $p\in X$ and a finite-dimensional representation $U$ of the Lie algebra ${\fr L}_+$ of vector fields of non-negative degree on an affine space.  
   Modules studied by Rudakov in \cite{Ru} correspond to 
$X=\bb{A}^n$, $p=0$ and simple modules $U$.

Gauge modules are generalizations of  \emph{modules of tensor densities} (or simply \emph{tensor modules}) whose construction was inspired by non-abelian gauge theory. 
  Tensor modules were the key objects in the classification theory for ${\mathcal W}_n$ (modules of \emph{intermediate series} in the case of ${\mathcal W}_1$) \cite{BF2}.   
  The algebra of functions $A_X$, the algebra of vector fields $\DD$, and the space of $1$-forms are natural examples of tensor modules. 
Let $h$ be a standard minor of the Jacobian matrix of the defining ideal of $X$. Denote by  $A_{(h)}$ the localization of $A$ by $h$.  
  We define \emph{gauge} $A\D$-modules as submodules of $A_{(h)}\otimes U$ for each chart, where $U$ is  a finite-dimensional ${\fr L}_+$-module, and the action of 
$\DD$ involves gauge fields $\{ B_i \}$.
 We will say that such gauge modules are associated with $U$. 
The \emph{tensor modules} defined in~\cite{BN} are examples of gauge modules where the functions $B_i$ are all zero.  In particular, we obtain classical \emph{modules of tensor densities} when $X$ is the torus and $\mathcal{W}_n$ is the derivation algebra of Laurent polynomials.

We expect a family of gauge modules to be quite large as indicated in the following conjecture. 
 
 \begin{con}
Every $A\D_{X}$-module that is finitely generated over $A_X$ is a gauge module.
\end{con}

Our main result is the following

\
\

\noindent{\bf Main Theorem.}\label{thm-introd}
\emph{Let ${\bf k}$ be an algebraically closed field of characteristic $0$,  $X \subset \A_{\bf k}^n$  an irreducible affine variety of dimension $s$,  
 $U$  a finite-dimensional simple $\fr{gl}_{s}({\bf k})$-module.  Then
\begin{itemize}
\item The Rudakov module $R_{p}(U)$ is a simple $A\D$-module for any non-singular point $p\in X$.
\item If $X$ is smooth then any gauge $A\D$-module associated with $U$ is simple.
\end{itemize}
}
\
\

This result allows us to construct new families of simple $A\D$-modules. We note that the question of simplicity of restrictions of Rudakov and gauge modules to the Lie algebra 
$\DD$ remains open. We are going to address this question in a subsequent paper.

\subsection*{Acknowledgements}
The present paper is based on the work conducted during the visit of Y.B. and J.N. to University of S\~{a}o Paulo. 
This visit was partially supported through SPRINT grant funded by FAPESP (2016/50475-3) and by Carleton University. Y.B. and J.N. gratefully acknowledge the hospitality of the University of S\~{a}o Paulo. Y.B. acknowledges support from the Natural Sciences and Engineering Research Council of Canada. V.F. was
supported in part by  CNPq grant (301320/2013-6) and by  FAPESP grant (2014/09310-5).

\section*{Preliminaries}
\subsection*{Functions and vector fields on algebraic varieties}
Our general setup follows the papers~\cite{BF3} and~\cite{BN} where more details can be found. We reiterate the basics of the setup here.

Let $X\subset \bb{A}^{n}$ be an irreducible affine algebraic variety over an algebraically closed field ${\bf k}$ of characteristic zero, and let $I_X = \left< g_1, \ldots, g_m \right>$ be the ideal of all functions that vanish on $X$. 
Let $A_X:={\bf k}[x_1, \ldots, x_n] / I_X$ be the algebra of polynomial functions on $X$.
 Denote by $\D_X:=\Der_{\bf k}(A_X)$  the Lie algebra of polynomial vector fields on $X$. We shall often drop the subscripts and write just $A$ and $\D$ for $A_X$ and $\D_X$. Note that $\D$ is an $A$-module and that $A$ is a left $\D$-module.
We can give a more explicit description of the Lie algebra $\D$ using Lie algebra $W_n$ of vector fields on  $\bb{A}^{n}$, $W_n:=\Der({\bf k}[x_1,\ldots, x_n])$. 
It was shown in ~\cite{BF3} that there is an isomorphism of Lie algebras:
\[\D \simeq \{\mu \in W_n \; | \; \mu(I_X)\subset I_X\} / \{\mu \in W_n \; | \; \mu({\bf k}[x_1,\ldots,x_n])\subset I_X\}.\]

Alternatively, we can consider $\D$ as a subalgebra of $\bigoplus_{i=1}^n A\frac{\del}{\del x_i}$. If we define a matrix $J=(\frac{\del g_i}{\del x_j})_{i,j}$ and consider it as a map $J:A^{\oplus n} \ra A^{\oplus m}$ then $\sum_{i=1}^{n} f_i \frac{\del}{\del x_i} \in \D$ if and only if $(f_1, \ldots, f_n) \in \mathrm{Ker}\; J$ ~\cite{BF3}.

Let $r:=\rank_{F} J$ where $F$ is the field of fractions of $A$ and let 
$\{h_i\}$ be the nonzero $r\times r$-minors of $J$. Define charts $N(h_i):=\{ p\in X\; |\; h_i(p) \neq 0\}$. If $X$ is smooth, these charts cover $X$ and we call this set of charts the \emph{standard atlas} for $X$.

Recall from \cite{BN} that $t_1, \ldots, t_s \in A$ are called \emph{chart parameters} in the chart $N(h)$ if the following conditions hold:
\begin{itemize}
	\item $t_1,\ldots, t_s$ are algebraically independent over ${\bf k}$, so that ${\bf k}[t_1,\ldots, t_s] \subset A$.
	\item Every $f\in A$ is algebraic over ${\bf k}[t_1,\ldots, t_s]$.
	\item The derivation $\frac{\del}{\del t_i}$ of ${\bf k}[t_1,\ldots, t_s]$ extends uniquely to a derivation of the localized algebra $A_{(h)}$.
\end{itemize}
From these conditions it also follows that $s=\dim X$ and that
$$\Der(A_{(h)}) = \bigoplus_{i=1}^s A_{(h)}\frac{\del}{\del t_i},$$ see~\cite{BN} for details. Since $\D = \Der(A) \subset \Der(A_{(h)})$, each vector field $\eta$ has a unique representation $\eta = \sum_{i=1}^s f_i \frac{\del}{\del t_i}$ for some $f_i \in A_{(h)}$. The \emph{standard chart parameters} in $N(h_i)$ are chosen to be the variables $x_k$ such that the $k$-th column of $J$ is not part of the minor $h_i$. 

We recall the following result from ~\cite[Section~3]{BF3}.
\begin{lemma}
Let $t_1, \ldots, t_s$ be standard chart parameters in the chart $N(h)$. Then $h \frac{\del}{\del t_i} \in \D$ for all $i$.
\end{lemma}

The following result was also stated in ~\cite{BF3}.
\begin{lemma}
\label{locpar}
Let $t_1, \ldots, t_s$ be standard chart parameters in the chart $N(h)$, and let $p\in N(h)$. Let $\overline{t}_{i}=t_i-t_i(p)$. Then $\overline{t}_1\ldots, \overline{t}_s$ are local parameters at $p$ in the classical sense of~\cite[Section~2.2.1]{S}.
\end{lemma}
\begin{proof}
We need to show that $\{\ol{t}_1, \ldots, \ol{t}_s\}$ is a basis for $\fr{m}_{p} / \fr{m}_{p}^2$. Clearly, $\ol{t}_i \in \fr{m}_{p}$. Since $s=\dim X$, it suffices to prove linear independence. Suppose that $\sum_{i=1}^s c_i\ol{t}_i \in \fr{m}_{p}^2$ for some $c_i \in {\bf k}$. Then $d\big( \sum_{i=1}^s c_i\ol{t}_i \big) \in \fr{m}_{p}$ for all derivations $d \in \Der{A}$. Taking $d=h\frac{\del}{\del t_k}$ we get $h(p)c_k 1(p) = 0 \Leftrightarrow c_k=0$ for all $k$, which shows that the set $\{\ol{t}_1, \ldots, \ol{t}_s\}$ is linearly independent in $\fr{m}_{p} / \fr{m}_{p}^2$, and thus it is a basis.
\end{proof}

\subsection*{$A\D$-modules}
We shall study spaces $M$ equipped with module structures over both the commutative unital algebra $A$ and over the Lie algebra $\D$ such that the two actions are compatible in the following sense:
\[\eta \cdot (f\cdot m) = \eta(f) \cdot m + f \cdot (\eta \cdot m)\]
for all $\eta \in \D$, $f\in A$, and $m\in M$. Equivalently, $M$ is a module over the smash product $A\#\U(\D)$, see~\cite{Mo} for details. For brevity we define $A\D:=A\#\U(\D)$. The category of $A\D$-modules is equipped with a tensor product: for $A\D$-modules $M$ and $N$, the space $M\otimes_{A} N$ is also an $A\D$-module, where we have $\eta \cdot (m\otimes n) := \eta \cdot m \otimes n + m\otimes \eta \cdot n$ as usual, see~\cite[Section~2]{BN} for details.

The category of $A\D$-modules is also equipped with duals. First of all, for $M \in A\D$-Mod we define
\[M^{*} = \Hom_{\bf k}(M,{\bf k})\]
to be the full dual space. Here a function $f$ acts by $(f \cdot \varphi)(m):=\varphi(f \cdot m)$, and a vector field $\eta$ acts by
$(\eta \cdot \varphi)(m) = -\varphi(\eta \cdot m)$. These actions are compatible, so $M^{*}$ is an $A\D$-module. If $\D$ possesses a Cartan subalgebra $\fr{h}$, we may also consider the \emph{restricted dual} of $M$; this is the submodule $\ol{M}^{*}:= \displaystyle\bigoplus_{\lambda \in \mathfrak{h}^{*}} \Hom_{\bf k}(M_{\lambda},{\bf k})$ of $M^{*}$, where $M_{\lambda}$ is a weight subspace of $M$ of weight $\lambda$ with respect to $\fr{h}$. 

On the other hand, we also define \[M^{\circ}:= \Hom_{A}(M,A).\]
We equip this space with the natural $A$-action $(f\cdot \varphi)(m)=\varphi(f\cdot m)$, and we define the action of a vector field $\eta$ by
\[(\eta \cdot \varphi)(m):= - \varphi(\eta \cdot m) + \eta(\varphi(m)).\]
These actions are also compatible so $M^{\circ}$ is an $A\D$-module.

Duals and tensor products allow us to construct more $A\D$-modules. In particular, the module of 1-forms may be defined as $\Omega^1_X = \D_X^\circ$.

\begin{ex}
Let $X=\Sp^1$ be the circle. Here $A={\bf k}[t,t^{-1}]$ and $\D$ is spanned by $\{e_k\}_{k\in \bb{Z}}$ where $e_k=t^{k+1}\frac{\del}{\del t}$.
For each $\alpha \in {\bf k}$ we have an $A\D$-module $\fr{F}_{\alpha}$ spanned by $\{v_s\}_{s\in \bb{Z}}$ where the action is given by
\[t^k \cdot v_s = v_{s+k} \qquad \text{ and } \qquad e_k \cdot v_s = (s+\alpha k)v_{k+s}.\]

In this setting we get the following relation between the different duals:
\[\fr{F}_{\alpha}^{\circ} \simeq \fr{F}_{-\alpha} \qquad \text{ and } \qquad \ol{\fr{F}}_{\alpha}^{*} \simeq \fr{F}_{1-\alpha}.\]
\end{ex}

\subsection*{Filtration of $\D$}
Fix a standard chart $N(h)$ with chart parameters $t_1, \ldots, t_s$ and fix a point $p$ in this chart.
Write $\fr{m}_p$ for the maximal ideal in $A$ consisting of functions that vanish at $p$. For $l \geq -1$, define $\D(l):=\{\eta\in \D \; | \; \eta(A) \in \fr{m}_{p}^{l+1}\}$. Then we have a filtration of subalgebras
\[\D=\D(-1) \supset \D(0) \supset \D(1) \supset \cdots, \]
with $[\D(l),\D(k)] \subset \D(l+k)$ for $l+k \geq -1$. This also shows that for $l \geq 0$, $\D(l)$ is an ideal of $\D(0)$. To simplify notation we shall sometimes write $\D_{+}$ for $\D(0)$.

\begin{lemma}
\label{filt}
We have $\D(l) = \fr{m}_{p}^{l+1}\D$.
\end{lemma}
\begin{proof}
It is clear that $\fr{m}_{p}^{l+1}\D \subset \D(l)$. For the reverse inclusion, take $\eta \in \D(l)$ and express it as $\eta = \sum_{i=1}^s f_i \frac{\del}{\del t_i}$. Then by the definition of $\D(l)$ we have $ h f_k = h\eta(t_k)\in \fr{m}_{p}^{l+1}$ for each $k$. Hence $f_k \in \fr{m}_{p}^{l+1}$ as $h(p) \neq 0$. Since $h\frac{\del}{\del t_i} \in \D$ we have $h\eta \in \fr{m}_{p}^{l+1}\D$.

But we also have $(h-h(p))^{l+1}\eta \in \fr{m}_{p}^{l+1}\D$. Expanding this and using the fact that $h\eta \in \fr{m}_{p}^{l+1}\D$ we also get $h(p)^{l+1}\eta \in \fr{m}_{p}^{l+1}\D$.
 Since $h(p)\neq 0$ we finally have $\eta \in \fr{m}_{p}^{l+1}\D$ which completes the proof.
\end{proof}

\section*{Rudakov modules}
Let $p$ be a non-singular point of $X$ and let $\{t_1, \ldots, t_s\}$ be the standard chart parameters centered at $p$, i.e. $t_1(p)=\cdots = t_s(p)=0$. In other words, given standard chart parameters $x_{j_1}, \ldots, x_{j_s}$, we take $t_i=x_{j_i}-x_{j_i}(p)$. 

Write $\fr{L}$ for the algebra of polynomial derivations, 
$$\fr{L}=\bigoplus_{i=1}^{s} {\bf k}[X_1,\ldots, X_s] \frac{\del}{\del X_i}.$$ If $Q$ is a monomial of degree $d$, we define the degree of the derivation $Q\frac{\del}{\del X_i}$ to be $d-1$. For $l \geq -1$, let $\fr{L}(l)$ the subalgebra of $\fr{L}$ consisting of derivations with no terms of degree less than $l$. We shall usually write $\fr{L}_{+}$ for $\fr{L}(0)$. This concept of degrees also extends to the Lie algebra
\[\hat{\fr{L}}:=\Der({\bf k}[[X_1,\ldots, X_s]]) = \bigoplus_{i=1}^{s} {\bf k}[[X_1,\ldots, X_s]] \frac{\del}{\del X_i},\]
and we have filtrations
\[\fr{L} \supset \fr{L}_{+} \supset \fr{L}(1)\supset \fr{L}(2) \supset \cdots\]
and
\[\hat{\fr{L}} \supset \hat{\fr{L}}_{+} \supset \hat{\fr{L}}(1)\supset \hat{\fr{L}}(2) \supset \cdots.\]

Consider the embedding 
$\D \subset \hat{\fr{L}}$ with $t_i \mapsto X_i$ discussed in~\cite[Section~3]{BF3}. It follows from Lemma~\ref{filt} that in this embedding we have
 $\D(l) = \hat{\fr{L}}(l) \cap \D$.

\begin{lemma}
\label{pow}
(a) For any element $\mu\in \fr{L}$ whose terms all have degree less than $N$,
there exists $\eta \in \D$ such that 
\[\eta = \mu + \text{ terms of degree }\geq N.\]

(b) There is an isomorphism of Lie algebras\[\D_{+} / \D(l) \simeq \fr{L}_{+} / \fr{L}(l).\]

(c) In particular, we have\[\D_{+} / \D(1) \simeq \fr{gl}_s({\bf k}).\]
\end{lemma}
\begin{proof}
For part $(a)$ it is sufficient to show that for every $N \in \bb{N}$ there exists $\eta \in \D$ such that
\[\eta = \frac{\del}{\del X_i} + \text{terms of degree } \geq N.\]
Then the claim of $(a)$ will follow since such vector fields may be multiplied by a polynomial in $X_1, \ldots, X_s$.

To construct $\eta$ we first take $h\frac{\del}{\del X_i} \in \D$. Since $h(p) \neq 0$, the power series for $h$ is invertible, and we can write
\[h^{-1} = q_N + \text{ terms of degree } > N,\]
with $q_N \in {\bf k}[X_1, \ldots, X_s] \subset A$. Then $\eta=q_Nh\frac{\del}{\del X_i}$ will have the desired form.

Part $(b)$ is an immediate consequence of $(a)$, and part $(c)$ follows from the fact that $\fr{L}_{+} / \fr{L}(1) \simeq \fr{gl}_s({\bf k})$. 
\end{proof}

Let $U$ be a finite-dimensional $\fr{L}_{+}$-module. By the discussion after ~\cite[Lemma~2]{Bi}, 
there exists $l \in \bb{N}$ such that $\fr{L}(l) U = (0)$, hence $U$ is an $\fr{L}_{+} / \fr{L}(l)$-module. The isomorphism of Lemma~\ref{pow} $(b)$ defines a $\D_{+}$-module structure on $U$ such that $\D(l)U=(0)$. We also define an $A$-action on $U$ by evaluation: $f \cdot u := f(p)u$ for $f\in A$ and $u\in U$. Note that $\fr{m}_{p}U = (0)$. For $\eta \in \D_{+}$ we have
\[\eta \cdot (f \cdot u) = f(p)\eta \cdot u =f \cdot (\eta \cdot u) + \eta(f) \cdot u,\]
 since $\eta(f) \in \fr{m}_{p}$. This shows that the two actions are compatible and that $U$ is in fact an $A\#\U(\D_{+})$-module.\\

The {\bf Rudakov module} $R_{p}(U)$ is defined as an induced module
\[R_{p}(U) := A\#\U(\D) \otimes_{A\#\U(\D_{+})} U.\]

\begin{rmk}
The special case when $X=\bb{A}^n$, $p=0$, and $U$ is a simple $\fr{gl}_n$-module was studied by Rudakov in~\cite{Ru}. The corresponding module $R_{0}(U)$ was shown to be simple as a $W_n$-module whenever $R_{0}(U)$ does not appear in the de Rham complex.
\end{rmk}

\begin{thm}
\label{rudakovsimple}
Let $U$ be a finite-dimensional simple $\D_{+} / \D(1)\simeq \fr{gl}_{s}({\bf k})$-module, and let $p$ be a non-singular point of $X$.
 Then the corresponding Rudakov module $R_{p}(U)$ is a simple $A\D_{X}$-module.
\end{thm}

To prove this theorem we need some  preliminary results.

We define a chain of subspaces in the Rudakov module by
$R_0 := 1\otimes U$ and $R_{i+1}:=R_i + \D \cdot R_i$. We also let $R_i:=(0)$ for $i<0$.
This gives a filtration
\[R_0 \subset R_1 \subset R_{2} \subset \cdots \quad \text{with} \quad  \cup_{i=0}^{\infty} R_i = R_{p}(U).\]
\begin{lemma}
\label{rels}
We have
\begin{enumerate}[label=(\alph*)]
	\item $\fr{m}_p R_l \subset R_{l-1}$.
	\item $\D(j)R_l \subset R_{l-j}$ for all $j$. 
\end{enumerate}
\end{lemma}
\begin{proof}
We proceed to prove these claims by induction on $l$. We first prove $(a)$. For $l=0$, claim $(a)$ obviously holds. For the inductive steps we note that
\[\fr{m}_p R_{l+1} \subset \fr{m}_pR_{l} + \D\fr{m}_{p}R_l+[\fr{m}_{p},\D]R_l.\]
Here the two first terms on the right side lie in $R_l$, and since $[\eta, f] = \eta(f)$ in the algebra $A\#\U(\D)$, the third term also lies in $AR_l = ({\bf k} \oplus \fr{m}_p)R_l = R_l$. Thus claim $(a)$ holds by induction.

For claim $(b)$ we first consider the base case $l=0$. Since $\D(1)U=0$ and $\D(0)U\subset U$, the base case is trivially true for $j\geq 0$. For $j=-1$, the base case holds by definition of the sequence $R_l$.

For the induction step we assume that for a fixed $l$ and for all $j\geq -1$ we have $\D(j)R_l \subset R_{l-j}$, and we compute
\[\D(j) R_{l+1} \subset \D(j)R_{l}+\D(j)\D R_l \subset \D(j)R_{l}+\D\D(j) R_l + [\D(j),\D]R_l.\]
\[\subset \D(j)R_{l}+\D R_{l-j} + \D(j-1)R_l \subset R_{(l+1)-j},\]
and claim $(b)$ also follows by induction.
\end{proof}

\begin{cor}
\label{locnil}
Both $\fr{m}_{p}$ and $\D(1)$ act locally nilpotent on $R_{p}(U)$.
\end{cor}

\begin{rmk}
It follows from the previous Corollary that for any $v \in  R_{p}(U)$ the space $Av$ is finite-dimensional. Hence the Rudakov module is not finitely generated as an $A$-module. 
\end{rmk}

\begin{prop}
\label{reduction}
Let $U$ be a finite-dimensional $\fr{L}_{+}$-module and let $R_{p}(U)$ be the corresponding Rudakov module. For any nonzero $v\in R_p(U)$ we have $Av \cap (1\otimes U) \neq (0)$.
\end{prop}
\begin{proof}
Pick $l$ such that $v \in R_{l}$. Consider the following elements of $\D$, given by Lemma~\ref{pow} $(a)$:
\[\eta_i = \frac{\del}{\del X_i} + \text{ terms of degree } \geq l.\]
Then $\{\eta_1, \ldots, \eta_s\}$ is a basis of the space $\D / \D_{+}$. Note that for any $w \in R_{l-2}$ we have
$\eta_i \eta_j w = \eta_j \eta_i w$ since $[\eta_i,\eta_j]w \in \D(l-1)R_{l-2} \subset R_{-1}=(0)$. Using the Poincar\'e-Birkhoff-Witt theorem and this commutativity relation we may write
\[v = \sum_{i=1}^{\dim U} P_i(\eta_1, \ldots, \eta_s)u_i,\]
where $P_i(\eta_1, \ldots, \eta_s)$ are polynomials of degree $\leq l$, and $\{u_i\}$ is a basis of $U$.

We claim that 
\[t_k \cdot v = -\sum_{i=1}^{\dim U} \big(\frac{\del}{\del \eta_k}P_i(\eta_1, \ldots, \eta_s)\big)u_i.\]
Indeed, $t_ku_i=0$ and $[t_k,\eta_i]=-\eta_i(t_k) = -\delta_{i,k} + a$ where $a\in \fr{m}_p^{l+1}$ and by Lemma~\ref{rels} $(a)$ we have $aR_l = 0$.

Now choose a polynomial $P$ among $\{P_i\}$ with maximal degree $d$, and let $\eta_1^{r_1}\cdots \eta_s^{r_s}$ be a monomial occurring in $P$ with nonzero coefficient and with $\sum r_i = d$.  Then the above discussion shows that $Av$ contains the nonzero element
\[t_1^{r_1} \cdots t_s^{r_s}\cdot v \in 1\otimes U.\]
This completes the proof. 
\end{proof}

We are now ready to prove the main theorem of this section.
\subsection*{Proof of Theorem \ref{rudakovsimple}}
To establish the simplicity of $R_p(U)$ we need to show that every non-zero vector $v\in R_{p}(U)$ generates $R_{p}(U)$ as an $A\D$-module.
By Proposition~\ref{reduction}, the $A$-submodule generated by $v$ contains a nonzero vector $u \in U$. Since $U$ is a simple $\D_{+}$-module, the $A\D$-submodule generated by $v$ contains $1\otimes U$. By construction of the Rudakov module $1\otimes U$ generates $R_p(U)$. This completes the proof of the theorem.

\section*{Gauge modules}
We use notation of the previous section.
Let $h$ be a standard minor in the Jacobian matrix $J$, 
 and let $(U,\rho)$ be a finite-dimensional $\fr{L}_{+}$-module.

\begin{defi} Functions $B_i: \ A_{(h)} \otimes U\ra A_{(h)} \otimes U$, $1\leq i \leq s$,
are called {\it gauge fields} if

(i) each $B_i$ is $A_{(h)}$-linear,

(ii) $[B_i,\rho(\fr{L}_{+})]=0$,

(iii) $[\frac{\del}{\del t_i}+B_i,\frac{\del}{\del t_j}+B_j]=0$ as operators on $A_{(h)}\otimes U$ for all $1\leq i,j \leq s$.
\end{defi}

\begin{lemma}
Let $(U,\rho)$ be a finite-dimensional $\fr{L}_{+}$-module and let $\{ B_i \}$ be gauge fields, $1\leq i \leq s$. Then the space $A_{(h)} \otimes U$ is an $A_{(h)} \Der A_{(h)}$-module with the following action of 
$\Der A_{(h)} = \bigoplus_{i=1}^{s}A_{(h)}\frac{\del}{\del t_{i}}$:
\begin{equation}
\label{gaction}
\tag{1} \left( f \frac{\del}{\del t_i} \right) \cdot (g\otimes u)=
f \frac{\del g}{\del t_i}\otimes u + g f B_i(1\otimes u)  
 +\sum_{k\in \bb{Z}_{+}^{s} \setminus \{0\}}\frac{1}{k!}g \frac{\del^{k} f}{\del t^k}\otimes \rho \left( X^k \frac{\del}{\del X_i} \right) u,
\end{equation}
where $g \in A_{(h)}$, $u\in U$, and $k!=\prod_{i=1}^s k_i!$ for $k = (k_1,\ldots, k_s)$. Note that the sum on the right is finite.
\end{lemma}

The proof of this lemma is a direct computation and we leave it as an exercise to the reader.

Identifying the Lie algebra vector fields $\D$ with its natural embedding into $\Der A_{(h)}$, we immediately obtain

\begin{lemma}
\label{actionlemma}
Together with the natural left $A$-action, the $\D$-action \eqref{gaction} above equips $A_{(h)} \otimes U$ with the structure of an $A\D$-module. 
\end{lemma}

\begin{defi}
An $A\D$-submodule of a module $A_{(h)} \otimes U$, which is finitely generated as an $A$-module, will be called a {\bf local gauge module}.
\end{defi}

\begin{rmk}
Note that simple $\fr{L}_{+}$-modules $U$ correspond to simple modules over $\fr{L}_{+} / \fr{L}(1) \simeq \fr{gl}_s$. 
In this case, the third term in (\ref{gaction}) takes the simpler form
\[\sum_{k=1}^s g \frac{\del f}{\del t_k} \otimes E_{ki} \cdot u.\]
If we additionally take all $B_i$ as zero we recover the $A\D$-modules studied in~\cite{BN}. \end{rmk}

\begin{defi}
We shall say that an $A\D$-module $M$ is a {\bf gauge module} if it is isomorphic to a local gauge module for each chart $N(h)$ in our standard atlas.
\end{defi}

The conjecture from the introduction states that any module in the category $A\D$-Mod which is finitely generated over $A$ is a gauge module.

\subsection*{Example: Gauge modules of rank one on the sphere}
In this section we prove some further results about gauge modules in the case of  the sphere. It turns out that the class of gauge modules is wider than the set of tensor modules constructed in the paper~\cite{BN}.

Let $X=\Sp^2 \subset \bb{A}^3$ given by the equation $x^2+y^2+z^2=1$. We also use the notations $(x_1,x_2,x_3)=(x,y,z)$.
The Lie algebra $\D_{\Sp^2}$ of vector fields on the sphere is generated by $\Delta_{12}$, $\Delta_{23}$, and $\Delta_{31}$ 
as an $A$-module, where $\Delta_{ij}=x_j\frac{\del}{\del x_i}-x_i\frac{\del}{\del x_j}$. These generators satisfy the relation
\[x_1\Delta_{23}+x_2\Delta_{31}+x_3\Delta_{12}=0.\]

 Consider the chart $N(z)$ where $t_1=x$ and $t_2=y$ are chart parameters. Let $U=\mathrm{span}(u_{\alpha})$ be the one dimensional 
$\fr{gl}_2$-module where the identity matrix acts as $\alpha$. Taking $B_1=B_2=0$, we obtain an $A\D$-module structure 
on the space $A_{(z)}\otimes u_{\alpha}$. These modules coincide with those constructed in the paper~\cite{BN} where 
it was shown that for $\alpha \in \bb{Z}$, the space
\[\fr{F}^{z}_{\alpha} := z^{-\alpha}A\otimes u_{\alpha} \subset A_{(z)}\otimes u_{\alpha}\]
is a proper $A\D$-submodule which is free of rank $1$ over $A$.

 The analogous construction goes through for the other two standard charts $N(x)$ and $N(y)$ yielding the modules $\fr{F}^{x}_{\alpha}$ and $\fr{F}^{y}_{\alpha}$. As discussed in~\cite{BN}, these three modules are isomorphic via chart transformation maps. For example, the isomorphism $\fr{F}^{z}_{\alpha} \ra \fr{F}^{x}_{\alpha}$ is given by:  
\[ 
f \otimes u_{\alpha} \mapsto (\tfrac{z}{x})^{\alpha}f\otimes u_{\alpha}.\]
We write just $\fr{F}_{\alpha}$ for this chart-independent version of the module, this is what was called a \emph{tensor module} in~\cite{BN}.

Now, $\fr{F}^{z}_{\alpha}$ is isomorphic to $A\otimes u_{\alpha}$ as vector spaces via the map
\[A\otimes u_{\alpha} \ra \fr{F}_{\alpha} \qquad f\otimes u_{\alpha} \mapsto z^{-\alpha}f\otimes u_{\alpha}.\]
This correspondence lets us transfer the $A\D$-module structure of $\fr{F}_{\alpha}$ to $A\otimes u_{\alpha}$. 

This module structure on $A\otimes u_{\alpha}$ in fact coincides with the local gauge module structure on $A\otimes u_{\alpha}$ in the chart with $h = z$
as defined in Lemma~\ref{actionlemma}, but where we now have $B_1=B_x=-\alpha x z^{-2}$ and $B_2=B_y=-\alpha y z^{-2}$. However, in this gauge module setting,  $\alpha$ is no longer required to be an integer, and we obtain a larger class of gauge modules
$\{ \fr{F}_{\alpha} \; | \; \alpha \in {\bf k} \}$. It turns out that the action can be expressed more simply in a chart-independent way as described in the following theorem.

\begin{thm}
For each $\alpha \in {\bf k}$ we have a $\D_{\Sp^2}$-action on the space $\fr{F}_{\alpha} = A \otimes u_{\alpha}$ given by
\[(f\Delta_{ij})\cdot (g\otimes u_{\alpha}) = f\Delta_{ij}(g)\otimes u_{\alpha} + \alpha g\Delta_{ij}(f)\otimes u_{\alpha}.\]
Together with the natural $A$-action, $\fr{F}_{\alpha}$ is a simple $A\D$-module which is isomorphic to a gauge module in the sense of Lemma~\ref{actionlemma} above.
\end{thm}
\begin{proof}
The formula for the action is an easy computation and follows from the discussion above. The fact that $\fr{F}_{\alpha}$ is simple follows from the following section. 
\end{proof}

\subsection*{Simplicity of gauge modules}
Let $X\subset \bb{A}^n$ be an irreducible algebraic variety of dimension $s$.
Fix a chart $N(h)$ in the standard atlas, and let $t_1, \ldots, t_s$ be chart parameters.
\begin{prop}
\label{closedop}
Let $M$ be an $A\D$-submodule of $A_{(h)} \otimes U$, where $U$ is a finite-dimensional $\fr{gl}_s$-module with weight basis $\{u_k | k \in \Gamma\}$.
Then for $\sum_{k \in \Gamma} g_{k} \otimes u_{k} \in M$ we also have $\sum_{k \in \Gamma} (hg_k\otimes E_{ij}\cdot u_{k})\in M$ for all $1 \leq i,j \leq S$. In other words, $M$ is invariant under the operators $h\otimes E_{ij}$ on $A_{(h)}\otimes U$.
\end{prop}
\begin{proof}
It suffices to prove the statement for a single term $g\otimes u$. For each vector field $\mu\in \D$ and for each function $f\in A$ we have
\[(f\mu)\cdot (g\otimes u)-f(\mu \cdot (g\otimes u)) \in M.\]
Taking $f=t_i$ and $\mu=h\frac{\del}{\del t_j}$ we obtain the desired element in $M$:
\[(t_ih\frac{\del}{\del t_j})\cdot (g\otimes u)-t_i(h\frac{\del}{\del t_j} \cdot (g\otimes u))\]
\[=\sum_{q=1}^s hg\frac{\del t_i}{\del t_q}\otimes E_{qj}\cdot u =  hg\otimes E_{ij}\cdot u.\]
\end{proof}

Note that minor $h$ defining the chart $N(h)$ gives rise to a filtration of $A_{(h)}\otimes U$:
\[\cdots \subset h^{k+1}A\otimes U \subset  h^{k} A\otimes U \subset h^{k-1} A\otimes U \subset \cdots\]

\begin{defi}
Let $M$ be an $A\D$-submodule of $A_{(h)}\otimes U$. 
\begin{itemize}
	\item We say that $M$ is {\bf bounded} if $M \subset  h^{j} A\otimes U$ for some $j$.
	\item We say that $M$ is {\bf dense} if $M \supset  h^{k} A\otimes U$ for some $k$.
\end{itemize}
\end{defi}

Note that $M$ is bounded if and only if $M$ is finitely generated as an $A$-module, since $A$ is noetherian.

\begin{prop}
\label{density}
Let $U$ be a finite-dimensional simple $\fr{gl}_s$-module. Then every nonzero $A\D$-submodule of $A_{(h)} \otimes U$ is dense.
\end{prop}
\begin{proof}
Let $M \subset A_{(h)} \otimes U$ be a nonzero submodule. Let
 \[I=\{f \in A \, | \, f (A \otimes U) \subset M\}\]
Then $I$ is an ideal of $A$. To show that $M$ is dense we need to show that $h^N\in I$ for some $N$.

Let $\Gamma$ be a weight basis for $U$. Let $v\in M$ and write this element in the form $v=\sum_{k \in \Gamma}f_{k}\otimes u_{k}$ with $f_{k} \in A_{(h)}$, in fact we shall assume that $f_k \in A$ (otherwise just multiply $v$ by a power of $h$). 

Fix an index $k_{0}$ such that $f_{k_0}$ is nonzero. The
Jacobson density theorem implies that for each $k\in \Gamma$ there exists 
$w_k\in \U(\fr{gl}_s)$ such that $w_ku_{k_0}=u_k$ and $w_ku_{i}=0$ for $i\neq k_0$. Fix an ordering among the $E_{ij}$ and express $w_k$ in the corresponding PBW-basis and let $r$ be the highest length of terms occurring in this expression of $w_k$.
For products in $\U(\fr{gl}_s)$ of length $t$ where $0\leq t \leq r$, define the correspondence
\[E_{i_1j_1}\cdots E_{i_tj_t} \mapsto h^{r-t}(h\otimes E_{i_1j_1})\cdots(h\otimes E_{i_tj_t}).\]
Here the right side is viewed as an element of $\mathrm{End}_{\bf k}(M)$ in accordance with Proposition~\ref{closedop}.
Then by construction the element corresponding to $w_k$ maps $v$ to $h^rf_0\otimes u_k$.
Here$f_0:=f_{k_0}$ and $r$ depends on $k$. Letting $N$ be the maximum of the $r$-values we conclude that $h^Nf_0\otimes u_k \in M$ for all $k\in \Gamma$ which means that $h^Nf_0(A\otimes U) \subset M$ and $h^Nf_0 \in I$.

We now aim to apply Hilbert's Nullstellensatz to the function $h$. Fix $p\in N(h)$. We need to show that there exists $f\in I$ with $f(p)\neq 0$. We had already found $h^Nf_0\in I$ so if $f_0(p)\neq 0$ we are done.
Otherwise, let $K$ be a positive integer such that $h^K B_i(U) \subset A\otimes U$ for all $i$ and consider the element
$h\frac{\del}{\del t_i} (h^{N+K}f_0\otimes u_{k}) \in M$. This expands as
\[(N+K)f_0h^{N+K}\frac{\del h}{\del t_i}\otimes u_{k} + h^{N+K+1}\frac{\del f_0}{\del t_i}\otimes u_{k} \]
\[+ h^{N+K+1}f_0 B_i(u_k)
+ h^{N+K}f_{0}\sum_{q=1}^{s} \frac{\del h}{\del t_q} \otimes E_{qi}\cdot u_{k}.\]

Now the first, third, and fourth terms lie in $h^Nf_0(A\otimes U) \subset M$, so we also get $h^{N+K+1}\frac{\del f_0}{\del t_i}\otimes u_{k} \in M$ for all $i$. This shows that we may replace $f_0$ by $\frac{\del f_0}{\del t_i}$ in the argument. 

There is some product $d$ of derivations with $d(f_0)(p) \neq 0$.
So acting repeatedly with vector fields of form $h\frac{\del}{\del t_i}$ as above we eventually obtain $h^{S} d(f_0)\in I$ for some large enough $S$, and $h^{S} d(f_0)$ is nonzero at $p$. Thus for every point $p\in N(h)$ we have found a function in $I$ which is nonzero at $p$. Thus we have shown the contrapositive of the following statement: $h(p)=0$ whenever $p$ is a common zero for $I$. By Hilbert's Nullstellensatz this implies that $h^N \in I$ for some $N$, which in turn means that $M$ is dense. 
\end{proof}

\begin{cor}
Let $A_{(h)}\otimes U$ be an $A\D$-module as in Lemma \ref{actionlemma}, where $U$ is a simple  $\fr{gl}_s$-module. Then there exists at most one simple $A\D$-submodule of $A_{(h)}\otimes U$.
\end{cor}
\begin{proof}
Let $M$ and $M'$ be simple submodules in $A_{(h)} \otimes U$. By  Proposition~\ref{density} both modules are dense, so they both contain $h^N A\otimes U$ for sufficiently large $N$. Thus $M \cap M'$ is a nonzero submodule of both $M$ and $M'$ so by simplicity we must have $M=M'$.
\end{proof}

\begin{thm}\label{thm-gauge-simple}
Let $X$ be a smooth irreducible affine algebraic variety and let $M$ be a gauge module which corresponds to a simple finite-dimensional $\fr{gl}_s$-module $U$. 
Then $M$ is a simple $A\D$-module. 
\end{thm}
\begin{proof}
Let $M'$ be a nonzero submodule of $M$ and define
\[I=\{f \in A \, | \, f M \subset M'\}\]
Then $I$ is an ideal and it does not depend on the chart we use. Let $\{h_i\}$ be the standard minors giving our atlas for $X$. Proposition~\ref{density} implies that there exist natural numbers $\{k_i\}$ such that $h_{i}^{k_i} \in I$ for all $i$. But since $X=\cup_{i} N(h_i)$, for each $p\in X$ we have $h_i(p)\neq 0$ for some index $i$. But then the set of common zeros 
is empty, and Hilbert's weak Nullstellensatz gives $1\in I$ \cite{S}. In view of the definition of $I$ this says that $M=M'$.
\end{proof}


Theorem \ref{rudakovsimple} and Theorem \ref{thm-gauge-simple} imply our Main Theorem.

\section*{Pairing between gauge modules and Rudakov modules}

Let $M$ be an $A\D$-module which is finitely generated over $A$, and let $p$ be a non-singular point of $X$. Define $U:=M/\fr{m}_pM$.

\begin{lemma}
The space $U$ is an $A\#\U(\D_{+})$-module.
\end{lemma}
\begin{proof}
We first verify that $\fr{m}_p M$ is a $\D_{+}$-submodule of $M$. Let $\mu \in \D_{+}$, $f\in \fr{m}_p$, and $m\in M$. Since $M$ is an $A\D$-module we have
\[\mu \cdot(f \cdot m )= \mu(f) \cdot m + f\cdot (\mu \cdot m).\]
Here $\mu(f) \in \fr{m}_p$ by the definition of $\D_{+}$, so the right side is clearly in $\fr{m}_p M$. But $\fr{m}_p M$ is also an $A$-submodule of $M$. Thus $\fr{m}_p M$ is an $A\#\U(\D_{+})$-submodule of $M$, and so is the quotient $U=M/\fr{m}_pM$.
\end{proof}

Note that $U$ is an evaluation module over $A$: we have $f \cdot u = f(p)u$.

\begin{lemma}
The module $U$ is finite-dimensional.
\end{lemma}
\begin{proof}
Let $u_1, \ldots, u_k$ generate $M$ over $A$. Then any $m\in M$ can be expressed as $m=f_1u_1+\ldots + f_ku_k$ for some $f_i \in A$. But then $\ol{m}=f_1(p)u_1+\ldots+f_k(p)u_k$ in the quotient $U$, which shows that the images of the $u_i$ span $U$.
\end{proof}

Let $U^{*} = \Hom_{\bf k}(U,{\bf k})$ be the dual space of $U$. This is an $A\D_{+}$-module with the standard dual actions of $A$ and of $\D_{+}$.

Write $\langle - , - \rangle$ for the natural pairing $U \times U^{*} \ra {\bf k}$ where $\langle u , \varphi\rangle = \varphi(u)$.
This pairing satisfies the following compatibility conditions for the actions of $f\in A$ and of $\eta \in \D_{+}$:
\[\langle u, f\cdot \varphi \rangle = \langle f\cdot u, \varphi \rangle = f(p)\langle u, \varphi \rangle\]
\[\langle u, \eta \cdot \varphi \rangle=-\langle \eta \cdot u, \varphi \rangle\]

Define a map $\tau: A\#\U(\D_{+}) \ra A\#\U(\D_{+})$ by requiring $\tau|_{A} = id$ and $\tau|_{\D_{+}}=-id$. Then $\tau$ extends uniquely to an anti-involution of $A\#\U(\D_{+})$. Then for $w\in  A\#\U(\D_{+})$ we have $\langle w\cdot u,\varphi\rangle=\langle u,\tau(w)\cdot\varphi\rangle$.

Now consider the canonical projection $\pi: M \ra U$ of $A\D_{+}$ modules. This gives rise to an $A\D_{+}$-morphism of the duals: $\pi^{*}: U^{*} \ra M^{*}$.
Consider the Rudakov module corresponding to the $A\D_{+}$-module $U^{*}$: \[R_p(U^{*}) =A\#\U(\D) \otimes_{A\#\U(\D_{+})}U^{*}.\]

\begin{prop}
The canonical $A\D_{+}$-homomorphism $\pi^{*}: U^{*} \ra M^{*}$ extends uniquely to an $A\D$-homomorphism $\ol{\pi}^{*}: R_p(U^{*}) \ra M^{*}$.
\end{prop}
\begin{proof}
This follows by the adjunction between induction and restriction:
\[\Hom_{A\D_{+}}(U^{*},M^{*}) \simeq \Hom_{A\D_{+}}(U^{*},\Res_{A\D_{+}}^{A\D} M^{*})\]
\[\simeq \Hom_{A\D}(\Ind_{A\D_{+}}^{A\D} U^{*}, M^{*})\simeq \Hom_{A\D}(R(U^{*}), M^{*})\]
\end{proof}

We summarize the results of the present section. 
\begin{thm}
Let $X$ be an algebraic variety and let $p$ be a non-singular point on $X$. Let $M$ be an $A\D_X$-module which is finitely generated over $A$. Define $U:=M / \fr{m}_pM$ and let
\[R_p(U^{*})=A\#\U(\D) \otimes_{A\#\U(\D_{+})} U^{*}\]
be the corresponding Rudakov module. 
Then there is a natural pairing between the modules $M$ and $R_p(U^{*})$ given by
\[\langle m , r \rangle = \ol{\pi}^{*}(r)(m),\]
where $\ol{\pi}^{*}$ is the canonical extension of the morphism $\pi^{*}: U^{*} \ra M^{*}$ to $R_p(U^{*})$.
This pairing satisfies 
\[\langle f\cdot m,r\rangle=\langle m,f\cdot r\rangle \qquad \text{ and } \qquad 
\langle \eta\cdot m,r\rangle=-\langle m,\eta\cdot r\rangle\]
 for all $f\in A$, $\eta \in \D$, $m\in M$, and $r\in R_p(U^{*})$. Equivalently, we have
\[\langle w\cdot m,r\rangle=\langle m,\tau(w)\cdot r\rangle\] for all $w\in A\#\U(\D)$, where $\tau$ is the natural anti-involution on $A\#\U(\D)$.
\end{thm}

\end{document}